\documentclass[]{interact}

\usepackage{epstopdf}
\usepackage[caption=false]{subfig}

\usepackage[numbers,sort&compress]{natbib}
\bibpunct[, ]{[}{]}{,}{n}{,}{,}

\makeatletter
\def\NAT@def@citea{\def\@citea{\NAT@separator}}
\makeatother

\usepackage{color}

\theoremstyle{plain}
\newtheorem{theorem}{Theorem}[section]
\newtheorem{lemma}[theorem]{Lemma}

\theoremstyle{definition}
\newtheorem{definition}[theorem]{Definition}

\theoremstyle{remark}
\newtheorem{remark}{Remark}

\newcommand\qbin[3]{\left[\begin{matrix} #1 \\ #2 \end{matrix} \right]_{#3}}

\begin{document}

\title{On classical orthogonal polynomials related to Hahn's operator}

\author{
\name{R. \'Alvarez-Nodarse\textsuperscript{a}, K. Castillo\textsuperscript{b}, D. Mbouna\textsuperscript{b}, J. Petronilho\textsuperscript{b}\thanks{CONTACT J. Petronilho. Email: josep@mat.uc.pt}}
\affil{\textsuperscript{a}IMUS, Universidad de Sevilla, Departamento de An\'alisis Matem\'atico, Apdo. 1160, E-41080, Sevilla, Spain; \textsuperscript{b}CMUC, Department of Mathematics, University of Coimbra, 3001-501, Coimbra, Portugal}
}

\maketitle

\begin{abstract}
Let ${\bf u}$ be a nonzero linear functional acting on the space of polynomials.
Let $\mathbf{D}_{q,\omega}$ be a Hahn operator acting on the dual space of polynomials.
Suppose that there exist polynomials $\phi$ and $\psi$, with $\deg\phi\leq2$ and $\deg\psi\leq1$,
so that the functional equation
$$
\mathbf{D}_{q,\omega}(\phi {\bf u})=\psi{\bf u}
$$
holds, where the involved operations are defined in a distributional sense. In this note we state necessary and sufficient conditions, involving only the coefficients of $\phi$ and $\psi$, such that ${\bf u}$ is regular, that is,
there exists a sequence of orthogonal polynomials with respect to ${\bf u}$.
A key step in the proof relies upon the fact that a distributional Rodrigues-type formula holds without assuming that ${\bf u}$ is regular.
\end{abstract}


\section{Introduction and main result}

Let $\mathcal{P}$ be the space of all polynomials with complex coefficients and let
$\mathcal{P}_n$ be its subspace of all polynomials of degree less than or equal to $n$ ($n=0,1,\ldots$).
The classical orthogonal polynomial sequences (OPS) of Hermite, Laguerre, Jacobi,
and Bessel, constitute the most studied class of OPS.
In the framework of regular orthogonality, these OPS are defined as orthogonal with respect to a
moment linear functional ${\bf u}:\mathcal{P}\to\mathbb{C}$ such that
there exist two nonzero polynomials $\phi\in\mathcal{P}_2$ and $\psi\in\mathcal{P}_1$
so that ${\bf u}$ satisfies the functional equation
\begin{equation}\label{Dphiu=psiu-veryC}
{\bf D}(\phi{\bf u})=\psi{\bf u}\;.
\end{equation}
Here, $\mathcal{P}^*$ being the (algebraic) dual space of $\mathcal{P}$,
the left multiplication of a functional ${\bf u}\in\mathcal{P}^*$
by a polynomial $\phi\in\mathcal{P}$, and the (distributional)
derivative of ${\bf u}\in\mathcal{P}^*$, are the functionals
$\phi{\bf u}\in\mathcal{P}^*$ and ${\bf D}{\bf u}\in\mathcal{P}^*$
defined, respectively, by
\begin{equation}\label{def-Du-continuous}
\langle \phi{\bf u},f\rangle:=\langle {\bf u},\phi f\rangle\;,\quad
\langle {\bf D}{\bf u},f\rangle:=-\langle {\bf u},f'\rangle\quad (f\in\mathcal{P})\;.
\end{equation}
Hermite and Laguerre functionals (corresponding to the Hermite and Laguerre OPS) appear in (\ref{Dphiu=psiu-veryC}) taking $\phi\equiv\mbox{\rm const.}\neq0$ and $\deg\phi=1$, respectively.
If $\deg\phi=2$ we obtain a Jacobi functional whenever the zeros of $\phi$ are distinct,
and a Bessel functional if $\phi$ has a double zero.
As fundamental references on this issue, we mention Maroni's works \cite{M1991,M1993,M1994}.
For the general theory of OPS (continuous and discrete) we refer the reader to the influential monographs
by Chihara \cite{C1978}, Ismail \cite{I2005}, Nikiforov, Suslov, and Uvarov \cite{NSU1991},
and Koekoek, Lesky, and Swarttouw \cite{KLS2010}.
We also mention here
the recent unpublished class notes \cite{P2018}
(where the emphasis in on the algebraic approach developed by Maroni).
A natural question arises:
if ${\bf u}$ is a nonzero linear functional defined on $\mathcal{P}$ satisfying (\ref{Dphiu=psiu-veryC}), with $\phi\in\mathcal{P}_2$ and $\psi\in\mathcal{P}_1$, and if at least one among $\phi$ and $\psi$ is not the zero polynomial,
to determine necessary and sufficient conditions, involving only the coefficients of $\phi$ and $\psi$, such that ${\bf u}$ is regular (i.e., there exists an OPS with respect to ${\bf u}$).
This question has been answered in the following

\begin{theorem}{\rm \cite[Lemma 2 and Theorem 2]{MP1994}}\label{MP-ITSF-Thm}
Let ${\bf u}\in\mathcal{P}^\prime\setminus\{{\bf 0}\}$. Suppose that
\begin{equation}\label{EqDistC1}
{\bf D}(\phi{\bf u})=\psi{\bf u}\;,
\end{equation}
where $\phi\in\mathcal{P}_2$, $\psi\in\mathcal{P}_1$,
and at least one of $\phi$ and $\psi$ is not the zero polynomial.
Write
$$
\phi(x):=ax^2+bx+c\,,\quad\psi(x):=dx+e\,,\quad d_n:=d+an\,,\quad e_n:=e+bn\;.
$$
($a,b,c,d,e\in\mathbb{C}$; $|a|+|b|+|c|+|d|+|e|\neq0$.)
Then, ${\bf u}$ is regular if and only if
\begin{equation}\label{EqDistC2}
d_n\neq0\,,\quad \phi\Big(-\frac{e_n}{d_{2n}}\Big)\neq0\;,\quad 
\forall n\in\mathbb{N}_0\;.
\end{equation}
Under these conditions, the monic OPS $(P_n)_{n\geq0}$
with respect to ${\bf u}$ satisfies the three-term recurrence relation
\begin{equation}\label{ttrrC1}
P_{n+1}(x)=(x-\beta_n)P_n(x)-\gamma_nP_{n-1}(x) \;,
\end{equation}
with $P_{-1}(x)=0$, being
\begin{equation}\label{EqDistC3}
\beta_n=\frac{ne_{n-1}}{d_{2n-2}}-\frac{(n+1)e_{n}}{d_{2n}}\, ,\quad
\gamma_{n+1}=-\frac{(n+1)d_{n-1}}{d_{2n-1}d_{2n+1}}\phi\Big(-\frac{e_n}{d_{2n}}\Big)
\quad(n=0,1,\ldots)\;.
\end{equation}
In addition, the following (distributional) Rodrigues formula holds
\begin{equation}\label{EqDistRod}
P_n{\bf u}=k_n\,{\bf D}^n\big(\phi^n{\bf u}\big)\;,\quad k_n:=\prod_{j=0}^{n-1}d_{n+j-1}^{-1}\quad(n=0,1,\ldots)\;.
\end{equation}
\end{theorem}

The aim of this contribution is to state a $(q,\omega)-$analogue of Theorem \ref{MP-ITSF-Thm}, replacing in the functional equation (\ref{EqDistC1}) the derivative operator ${\bf D}$ by an appropriate (distributional) Hahn's operator, denoted by ${\bf D}_{q,\omega}$.

Given complex numbers $q$ and $\omega$,
the (ordinary) Hahn's operator $D_{q,\omega}:\mathcal{P}\to\mathcal{P}$ is
\begin{equation}\label{def-Dqw}
D_{q,\omega}f(x):=\frac{f(qx+\omega)-f(x)}{(q-1)x+\omega}\quad (f\in\mathcal{P})\;.
\end{equation}
This operator has been studied by Hahn \cite{H1949}.
Hereafter (when referring to $D_{q,\omega}$) we will assume that $q$ and $\omega$ fulfill the conditions
\begin{equation}\label{q-notexp}
|q-1|+|\omega|\neq0\;,\quad
q\not\in\big\{ 0,{\rm e}^{2ij\pi/n}\;|\;1\leq j\leq n-1\;;\;\;n=2,3,\ldots\big\}\;.
\end{equation}
The first condition in (\ref{q-notexp}) ensures that the right-hand side of (\ref{def-Dqw}) is well defined.
The second one is imposed in order to ensure the existence of OPS in Hahn's sense
(this will be made clear later --- cf. Theorem \ref{Dqw-main-Thm}).
The (ordinary) Hahn's operator $D_{q,\omega}$ induces a (distributional) Hahn's operator
${\bf D}_{q,\omega}:\mathcal{P}^*\to\mathcal{P}^*$, defined by
\begin{equation}\label{def-Dqwu}
\langle{\bf D}_{q,\omega}{\bf u},f\rangle:=
-q^{-1}\langle{\bf u},D_{q,\omega}^*f\rangle\quad({\bf u}\in\mathcal{P}^*\;,\;f\in\mathcal{P})\;,
\end{equation}
where $D_{q,\omega}^*:=D_{1/q,-\omega/q}$.
This definition of ${\bf D}_{q,\omega}$ appears in
Foupouagnigni's PhD thesis \cite[Definition 3.4]{F1998}.
A slightly different one was considered in
H\"acker's PhD thesis \cite[(1.16)]{H1993} (under the supervision of P. Lesky and reviewed for AMS by R. Askey),
where the adopted definition is $\langle{\bf D}_{q,\omega}{\bf u},f\rangle=-\langle{\bf u},D_{q,\omega}f\rangle$,
as it may seem more natural {\it a priori}, taking into account the standard definition appearing in
(\ref{def-Du-continuous}) for the continuous case.
The main results appearing in this thesis can be found also in \cite{H1993a}.
The advantage of (\ref{def-Dqwu}) stems from the facts pointed out in Remark \ref{defDqwboa}.
We also need the operators
$L_{q,\omega}:\mathcal{P}\to\mathcal{P}$ and ${\bf L}_{q,\omega}:\mathcal{P}^*\to\mathcal{P}^*$ given by
$$
L_{q,\omega}f(x):=f(qx+\omega)\;,\quad
\langle {\bf L}_{q,\omega}{\bf u},f\rangle:=\langle {\bf u},L_{q,\omega}^*f\rangle
\quad \big(f\in\mathcal{P}\;,\;{\bf u}\in\mathcal{P}^*\big)\;,
$$
where $L_{q,\omega}^*:=L_{1/q,-\omega/q}$.
Recall that the $q-$bracket is defined by
$$
[\alpha]_q:=
\left\{\begin{array}{cl}
\displaystyle\frac{q^\alpha-1}{q-1}\;,&\mbox{\rm if}\quad q\neq1 \\ [0.75em]
\alpha\;,&\mbox{\rm if}\quad q=1
\end{array}
\right.\quad(\alpha,q\in\mathbb{C})\;.
$$
Note that for each nonnegative integer number $n$, we have $[0]_q:=0$ and $[n]_q\to n$ as $q\to1$.
Note also that (\ref{q-notexp}) ensures that $[n]_q\neq0$ for each $n=1,2,\ldots$.
Our main result is the following:

\begin{theorem}\label{Dqw-main-Thm}
Fix $q,\omega\in\mathbb{C}$ fulfilling $(\ref{q-notexp})$.
Let ${\bf u}\in\mathcal{P}^\prime\setminus\{{\bf 0}\}$. Suppose that
\begin{equation}\label{EqDistC1-Dqw}
{\bf D}_{q,\omega}(\phi{\bf u})=\psi{\bf u}\;,
\end{equation}
where $\phi\in\mathcal{P}_2$, $\psi\in\mathcal{P}_1$,
and at least one of $\phi$ and $\psi$ is not the zero polynomial. Set
\begin{align}
& \phi(x):=ax^2+bx+c\,,\quad\psi(x):=dx+e\,, \label{abcpr1-Dqw} \\
& d_n\equiv d_n(q):=dq^n+a[n]_{q}\,,\quad e_n\equiv e_n(q,\omega):=eq^n+(\omega d_n+b)[n]_q\;. \label{abcpr2-Dqw}
\end{align}
Then, ${\bf u}$ is regular if and only if
\begin{equation}\label{EqDistC2-Dqw}
d_n\neq0\,,\quad \phi\Big(-\frac{e_n}{d_{2n}}\Big)\neq0\;,\quad \forall n\in\mathbb{N}_0\;.
\end{equation}
Under these conditions, the monic OPS $(P_n)_{n\geq0}\equiv(P_n(\cdot;q,\omega))_{n\geq0}$
with respect to ${\bf u}$ satisfies the three-term recurrence relation
\begin{equation}\label{ttrrC1-Dqw}
P_{n+1}(x)=(x-\beta_n)P_n(x)-\gamma_nP_{n-1}(x) \;,
\end{equation}
with $P_{-1}(x)=0$, being
\begin{align}
& \beta_n:=\omega[n]_q+\frac{[n]_qe_{n-1}}{d_{2n-2}}-\frac{[n+1]_qe_{n}}{d_{2n}}\, ,\label{EqDistC3-Dqw1} \\
& \gamma_{n+1}:=-\frac{q^{n}[n+1]_qd_{n-1}}{d_{2n-1}d_{2n+1}}\phi\Big(-\frac{e_n}{d_{2n}}\Big)
\quad(n=0,1,\ldots)\;. \label{EqDistC3-Dqw2}
\end{align}
In addition, the Rodrigues-type formula
\begin{equation}\label{EqDistRod-Dqw}
P_n{\bf u}=k_n\,{\bf D}_{1/q,-\omega/q}^n\Big(\Phi(\cdot;n){\bf L}_{q,\omega}^n{\bf u}\Big)
\quad(n=0,1,\ldots)
\end{equation}
holds in $\mathcal{P}^*$, where
\begin{equation}\label{Rod-knPhi}
k_n:=q^{n(n-3)/2}\prod_{j=0}^{n-1}d_{n+j-1}^{-1}\;,\quad
\Phi(x;n):=\prod_{j=1}^n\phi\big(q^jx+\omega[j]_q\big)\;.
\end{equation}
\end{theorem}

\begin{remark}
Under the assumption that ${\bf u}$ is regular, the Rodrigues-type formula (\ref{EqDistRod-Dqw}) appears in M\'edem et al. \cite{MA-NM2001} for $\omega=0$ and $q\neq1$, and in Salto \cite{S1995} for $q=1$ and $\omega\neq0$. However, we will prove a more general result (cf. Lemma \ref{lemmaRodFunctional}), showing that (\ref{EqDistRod-Dqw}) holds without assuming the regularity of ${\bf u}$, provided that $(P_n)_n$ is a simple set of polynomials defined by (\ref{ttrrC1-Dqw})--(\ref{EqDistC3-Dqw2}), which we see is well defined requiring only (the admissibility condition) $d_n\neq0$ for each $n=0,1,\ldots$. It is worth mentioning that this (non trivial) fact is known for the continuous case \cite[Lemma 2]{MP1994},
but for the $(q,\omega)-$case we did not found a reference in the available literature.
\end{remark}

\begin{remark}
Taking $\omega=0$ and letting $q\to1$ in Theorem \ref{Dqw-main-Thm} yields Theorem \ref{MP-ITSF-Thm}.
We highly that H\"acker \cite[Theorem 1.4 (p. 26)]{H1993} gave regularity conditions different from (\ref{EqDistC2-Dqw}),
considering a definition of ${\bf D}_{q,\omega}$ in the sense discussed above.
H\"acker's approach is very different from ours,
since his results are derived from the analysis of a discrete Sturm-Liouville problem,
while our proof of Theorem \ref{Dqw-main-Thm}
uses appropriate modifications of some ideas appearing in \cite{MP1994},
based in the McS thesis \cite{P1993} and obtained independently of H\"acker's results.
Indeed, our approach is supported on the algebraic theory of orthogonal polynomials developed by Maroni \cite{M1991}.
\end{remark}

\begin{remark}\label{defDqwboa}
As we mentioned before, there is some advantages in defining ${\bf D}_{q,\omega}$
as in (\ref{def-Dqwu}). For instance, in the regularity condition (\ref{EqDistC2-Dqw}) as well as
in the expression for $\gamma_{n}$ given by (\ref{EqDistC3-Dqw2}),
the polynomial appearing therein is precisely $\phi$.
The same does not holds in the formulas given in H\"acker thesis 
(cf. \cite[Section 2.4]{H1993}).
\end{remark}

In the next section some background needed throughout this work is introduced.
Section 3 is devoted to the proof of Theorem \ref{Dqw-main-Thm}.%

\section{Basic results and notations}

We start by recasting some basic definitions.
For $a\in\mathbb{C}\setminus\{0\}$ and $b\in\mathbb{C}$,
the dilation operator $h_a:\mathcal{P}\to\mathcal{P}$ and
the translation operator $\tau_b:\mathcal{P}\to\mathcal{P}$ are defined by
\begin{equation}\label{hataub}
h_af(x):=f(ax)\;,\quad \tau_bf(x):=f(x-b)\quad(f\in\mathcal{P})\,.
\end{equation}
Note that if $q=1$ in (\ref{def-Dqw}) then, setting $\triangle_\omega f(x):=f(x+\omega)-f(x)$, we have
$$
D_{1,\omega}=\frac{\triangle_\omega}{\omega}\;,\quad
$$
while if $q\neq1$ then, setting $D_q:=D_{q,0}$, we have 
\begin{equation}\label{DqDqw}
D_{q,\omega}=\tau_{\omega_0}D_{q}\tau_{-\omega_0}\;,\quad \omega_0:=\omega/(1-q)
\end{equation}
(see e.g. \cite[(7.1)]{CP2015}).
Thus, if $q\neq1$ then there is no loss of generality by assuming $\omega=0$, a fact remarked by Hahn himself \cite{H1949}.
Despite this, it seems to us preferable to present the theory for general $(q,\omega)$ fulfilling (\ref{q-notexp}),
in order to emphasize that there is no significant simplification by presenting it for specific $q$ or $\omega$
and, more interesting, there is no need to study separately the case $q=1$ and $q\neq1$.
As a matter of fact, the general formulas appearing in Theorem \ref{Dqw-main-Thm}
allow us to emphasize a complete similarity with the corresponding ones for the continuous case
(appearing in Theorem \ref{MP-ITSF-Thm}).

Next we introduce some basic definitions and useful notations.

\begin{definition}\label{def-Fqwu-Foup}
Let $q\in\mathbb{C}\setminus\{0\}$ and $\omega\in\mathbb{C}$.
\begin{enumerate}
\item[{\rm (i)}]
The operator $L_{q,\omega}:\mathcal{P}\to\mathcal{P}$ is defined by
$$
L_{q,\omega}:=h_q\circ\tau_{-\omega}\;.
$$
\item[{\rm (ii)}]
The operators $L_{q,\omega}^*:\mathcal{P}\to\mathcal{P}$ and $D_{q,\omega}^*:\mathcal{P}\to\mathcal{P}$ are defined by
$$
L_{q,\omega}^*:=h_{1/q}\circ\tau_{\omega/q}=L_{1/q,-\omega/q}\;,\quad
D_{q,\omega}^*:=D_{1/q,-\omega/q}\;.
$$
\item[{\rm (iii)}]
The operators ${\bf D}_{q,\omega}:\mathcal{P}^*\to\mathcal{P}^*$ and 
${\bf L}_{q,\omega}:\mathcal{P}^*\to\mathcal{P}^*$ are defined by 
$$
\langle{\bf D}_{q,\omega}{\bf u},f\rangle:=-q^{-1}\langle{\bf u},D_{q,\omega}^*f\rangle\;,\quad
\langle{\bf L}_{q,\omega}{\bf u},f\rangle:=q^{-1}\langle{\bf u},L_{q,\omega}^*f\rangle\quad
({\bf u}\in\mathcal{P}^*\;,\;f\in\mathcal{P})\;.
$$
\item[{\rm (iv)}]
The operators ${\bf D}_{q,\omega}^*:\mathcal{P}^*\to\mathcal{P}^*$ and
${\bf L}_{q,\omega}^*:\mathcal{P}^*\to\mathcal{P}^*$ are defined by
$$
{\bf D}_{q,\omega}^*:={\bf D}_{1/q,-\omega/q}\;,\quad
{\bf L}_{q,\omega}^*:={\bf L}_{1/q,-\omega/q}\;.
$$
\end{enumerate}
\end{definition}

\begin{remark}
As far as we know, the definitions appearing in (i), (ii), and (iv) were given in \cite{H1993},
while the ones appearing in (iii) were proposed in \cite{F1998}.
\end{remark}

Note that $L_{q,\omega}$ and $L_{q,\omega}^*$ are linear operators, given explicitly by
$$
L_{q,\omega}f(x)=f(qx+\omega)\;,\quad L_{q,\omega}^*f(x)=f\Big(\frac{x-\omega}{q}\Big)\qquad(f\in\mathcal{P})\;.
$$
In bellow we summarize some useful properties involving the above operators, where
${\bf u}\in\mathcal{P}^*$ and $f,g\in\mathcal{P}$
(see \cite{H1993,F1998,KLS2010}):
\begin{align}
&L_{q,\omega}^* L_{q,\omega} = L_{q,\omega}L_{q,\omega}^*=I\;;\quad
{\bf L}_{q,\omega}^* {\bf L}_{q,\omega} = {\bf L}_{q,\omega}{\bf L}_{q,\omega}^*=\mbox{\bf I}\;; \label{P2} \\
&L_{q,\omega}^{-1}=L_{q,\omega}^*\;; \quad
{\bf L}_{q,\omega}^{-1}={\bf L}_{q,\omega}^*\;; \label{P2a} \\
&L_{q,\omega}^nf(x) =f\big(q^nx+\omega[n]_q\big)\quad (n=0,\pm1,\pm2,\ldots)\;;  \label{P1}\\
&D_{q,\omega}^* D_{q,\omega} =q D_{q,\omega} D_{q,\omega}^*\;; 
\; D_{q,\omega} L_{q,\omega}^* =q^{-1} L_{q,\omega}^* D_{q,\omega} \;;\;
D_{q,\omega} L_{q,\omega} =q L_{q,\omega} D_{q,\omega} \;; \label{P3} \\
&D_{q,\omega}^* L_{q,\omega} =q D_{q,\omega} \;;\quad
{\bf D}_{q,\omega}^* {\bf L}_{q,\omega} =q {\bf D}_{q,\omega}\;; \label{P4} \\
&L_{q,\omega}(fg) = (L_{q,\omega}f)(L_{q,\omega}g)\;;\quad
{\bf L}_{q,\omega}(f{\bf u})=L_{q,\omega}f\,{\bf L}_{q,\omega}{\bf u}\label{P4a}\;; \\
&D_{q,\omega}(fg) = (D_{q,\omega}f)(L_{q,\omega}g)  + f D_{q,\omega}g  \label{P5} \\
&{\bf D}_{q,\omega}(f{\bf u}) = D_{q,\omega}f\,{\bf L}_{q,\omega}{\bf u} + f {\bf D}_{q,\omega}{\bf u}
=D_{q,\omega}f\,{\bf u}+L_{q,\omega}f\,{\bf D}_{q,\omega}{\bf u}\,. \label{P6}
\end{align}
(In (\ref{P2}), $I$ and ${\bf I}$ denote the identity operators
in $\mathcal{P}$ and in $\mathcal{P}^*$, respectively.)
We also point out the following analogue of Leibnitz formula:
\begin{equation}\label{Leibnitz-Dqwnfg}
D_{q,\omega}^n(fg)
=\sum_{k=0}^n \qbin{n}{k}{q}
L_{q,\omega}^k\big(D_{q,\omega}^{n-k}f\big)\cdot D_{q,\omega}^kg
\quad(f,g\in\mathcal{P})\;,
\end{equation}
where, defining the $q-$factorials as $[0]_q!:=1$ and $[n]_q!:=[1]_q[2]_q\cdots [n]_q$ for $n\in\mathbb{N}$,
the $q-$binomial number is given by
$$
\qbin{n}{k}{q}:=\frac{[n]_q!}{[k]_q![n-k]_q!}\quad (n,k\in\mathbb{N}_0\;;\;k\leq n)\;.
$$
Note that (\ref{Leibnitz-Dqwnfg}) can be easily deduced from the well known Leibnitz formula
for the operator $D_q$ (see e.g. \cite[Exercise 12.1]{I2005}
or \cite[(1.15.6)]{KLS2010}) and using the relation (\ref{DqDqw}) between $D_q$ and $D_{q,\omega}$.
There is also a functional version of the Leibnitz formula:
\begin{equation}\label{Leibnitz-Dqwnfu}
{\bf D}_{q,\omega}^n(f{\bf u})
=\sum_{k=0}^n \qbin{n}{k}{q} 
L_{q,\omega}^k\big(D_{q,\omega}^{n-k}f\big)\, {\bf D}_{q,\omega}^k{\bf u}
\quad(f\in\mathcal{P}\;,\;{\bf u}\in\mathcal{P}^*)\;.
\end{equation}
A basic property of Hahn's operator relies upon the fact it maps a polynomial of degree $n$ into one of degree $n-1$.
Indeed, since
$D_{q,\omega}x^{n}=\sum_{k=0}^{n-1}(qx+\omega)^kx^{n-1-k}$,
applying the binomial formula to $(qx+\omega)^k$, we obtain
\begin{equation}\label{Dqwxn}
D_{q,\omega}x^{n}=\sum_{k=0}^{n-1}[n,k]_{q,\omega}x^{n-1-k}=[n]_qx^{n-1}+(\mbox{\rm lower degree terms})\;,
\end{equation}
where the number $[n,k]_{q,\omega}$ is defined by
$$
[n,k]_{q,\omega}:=\omega^k\sum_{j=0}^{n-1-k}\binom{k+j}{j}q^j\quad (n,k=0,1,\ldots)\;.
$$
We adopt the convention that an empty sum equals zero, 
hence
$$
[n,k]_{q,\omega}=0\quad\mbox{\rm if}\quad n\leq k\;.
$$
We also point out the following useful representations:
$$
[n,k]_{q,\omega}=\frac{\omega^k}{k!}\frac{{\rm d}^k}{{\rm d}q^k}\left(\sum_{j=k}^{n-1}q^j\right)
=\frac{\omega^k}{k!}\frac{{\rm d}^k}{{\rm d}q^k}\left(\frac{q^{n}-q^k}{q-1}\right) 
=\frac{\omega^k}{k!}\frac{{\rm d}^k}{{\rm d}q^k}\Big([n]_q-[k]_q\Big)\;. 
$$
In particular, for $k\in\{0,1,2\}$, we compute
\begin{align*}
&[n,0]_{q,\omega} = [n]_q\;, \\
&[n,1]_{q,\omega} = \big(n[n-1]_q-(n-1)[n]_q\big)\omega_0\;, \\
&[n,2]_{q,\omega} = \big(n(n-1)[n-2]_q-2n(n-2)[n-1]_q+(n-2)(n-1)[n]_q\big)\omega_0^2/2\;,
\end{align*}
where $\omega_0$ is given by (\ref{DqDqw}).
Taking $\omega=0$ in (\ref{Dqwxn}) we see that $D_q$ fulfills
\begin{equation}\label{Dqxn}
D_{q}x^{n}=[n]_qx^{n-1}\quad(n=0,1,\ldots)\;.
\end{equation}
The usefulness of this property relies upon the following fact:
if ${\bf u}\in\mathcal{P}^*$, $\phi\in\mathcal{P}_2$, and $\psi\in\mathcal{P}_1$,
then ${\bf u}$ satisfies the functional equation $D_q(\phi{\bf u})=\psi{\bf u}$ if and only if the sequence of
moments $(u_n:=\langle{\bf u},x^n\rangle)_{n\geq0}$ satisfies the homogeneous second order linear difference equation
\begin{equation}\label{uDqxn}
d_n(q)u_{n+1}+e_n(q)u_{n}+f_n(q)u_{n-1}=0\quad(n=0,1,\ldots)\;,
\end{equation}
where $d_n(q)$, $e_n(q)$, and $f_n(q)$ are complex numbers.
Of course, taking into account (\ref{Dqwxn}),
the analogous to property (\ref{Dqxn}) no longer holds true if $D_{q}$ is replaced by $D_{q,\omega}$ ($\omega\neq0$).
Hence, one can not expect that the moments corresponding to a functional ${\bf u}$ fulfilling $D_{q,\omega}(\phi{\bf u})=\psi{\bf u}$ ---being ${\bf u}$, $\phi$, and $\psi$ as above--- satisfy a second order difference equation like (\ref{uDqxn}).
H\"acker replaced the power basis $(x^n)_{n\geq0}$ by a different polynomial basis,
$(X_n)_{n\geq0}\equiv(X_n(\cdot;q,\omega))_{n\geq0}$, chosen so that
\begin{equation}\label{DqwXn}
D_{q,\omega}X_n=\alpha_nX_{n-1}\quad(n=1,2,\ldots)\;,
\end{equation}
for suitable $\alpha_n\equiv\alpha_n(q,\omega)\in\mathbb{C}\setminus\{0\}$.
This is achieved by choosing \cite[p. 12]{H1993}
\begin{equation}\label{Xnqw1}
X_0:=1\;,\quad
X_{n}:=f\, H_{q,\omega}(f)\, H^2_{q,\omega}(f)\cdots H^{n-1}_{q,\omega}(f)\quad(n=1,2,\ldots)\;,
\end{equation}
where $H_{q,\omega}:=L_{q,\omega}^*$ and $f(x):=x$.
For this sequence $(X_n)_{n\geq0}$, H\"acker \cite[Lemma C.4]{H1993}
shown that the number $\alpha_n$ appearing in (\ref{DqwXn}) is given explicitly
by $\alpha_n:=q^{1-n}[n]_q$ (so, indeed, it does depend on $q$ and not on $\omega$).
It is easy to see that
$$
X_{n+1}(x)=q^{-n}\big(x-\omega[n]_q\big)X_{n}(x)\quad(n=0,1,\ldots)\;,
$$
and so we arrive at the explicit expression
$$
X_{n}(x)=q^{-\binom{n}{2}}\prod_{j=0}^{n-1}\big(x-\omega[j]_q\big)\quad(n=0,1,\ldots)\;.
$$
For our purposes it is more convenient to use a basis (of $\mathcal{P}$) of monic polynomials,
namely $(Y_n)_{n\geq0}\equiv(Y_n(\cdot;q,\omega))_{n\geq0}$, where $Y_n:=q^{\binom{n}{2}}X_{n}$,
so that 
$$
Y_0(x)=1\;,\quad
Y_{n+1}(x)=\big(x-\omega[n]_q\big)Y_n(x)=\prod_{j=0}^{n}\big(x-\omega[j]_q\big)\quad(n=0,1,\ldots)\;.
$$
Clearly, $(Y_n)_{n\geq0}$ fulfills the desired property:
\begin{equation}\label{DqwFn}
D_{q,\omega}Y_n(x)=[n]_qY_{n-1}(x)\quad(n=1,2,\ldots)\;.
\end{equation}
Finally, using (\ref{DqwFn}) it is straightforward to show that ${\bf u}\in\mathcal{P}^*$
satisfies the functional equation $D_{q,\omega}(\phi{\bf u})=\psi{\bf u}$
(being $\phi\in\mathcal{P}_2$ and $\psi\in\mathcal{P}_1$) if and only if the sequence of
moments with respect to the basis $(Y_n)_{n\geq0}$, $(y_n:=\langle{\bf u},Y_n\rangle)_{n\geq0}$,
fulfills
\begin{equation}\label{uDqwxn}
d_ny_{n+1}+\big(e_n+\omega[n]_qd_{n-1})\big)y_{n}+[n]_q(c+\omega e_{n-1})y_{n-1}=0\quad(n=0,1,\ldots)\;,
\end{equation}
where $d_n\equiv d_n(q)$ and $e_n\equiv e_n(q,\omega)$ are defined as in (\ref{abcpr2-Dqw}).

\section{Proof of Theorem \ref{Dqw-main-Thm}}

\subsection{Preliminary results}

Given a nonnegative integer number $k$ and a monic polynomial $P_n$ of degree $n$,
we denote by $P_n^{[k]}\equiv P_n^{[k]}(\cdot;q,\omega)$ the monic polynomial of degree $n$ defined by
\begin{equation}\label{def-Pnkmonic}
P_n^{[k]}(x):=\frac{D_{q,\omega}^k P_{n+k}(x)}{\mbox{$\prod_{j=1}^{k}[n+j]_q$}}
=\frac{[n]_q!}{[n+k]_q!}\,D_{q,\omega}^k P_{n+k}(x)\quad(k,n=0,1,\ldots).
\end{equation}
(If $k=0$, it is understood that $D_{q,\omega}^0f=f$ and that empty product equals one.)
We assume that ${\bf u}\in\mathcal{P}^*$ satisfies the functional equation
\begin{equation}\label{2.1}
{\bf D}_{q,\omega}(\phi {\bf u})=\psi \textbf{u}\;,
\end{equation}
where $\phi\in\mathcal{P}_2$ and $\psi\in\mathcal{P}_1$. Set
\begin{equation}\label{def-unk}
{\bf u}^{[0]}:={\bf u}\;,\quad
{\bf u}^{[k]}:={\bf L}_{q,\omega}\big(\phi{\bf u}^{[k-1]}\big)
=L_{q,\omega}\phi\,{\bf L}_{q,\omega}{\bf u}^{[k-1]}\quad(k=1,2,\ldots)\;,
\end{equation}
where the last equality holds by (\ref{P4a}).
Iterating (\ref{def-unk}) and taking into account (\ref{P1}) yields
\begin{equation}\label{repres-unk}
{\bf u}^{[k]}=\Big(\prod_{j=1}^kL_{q,\omega}^j\phi\Big){\bf L}_{q,\omega}^k{\bf u}
=\Phi(\cdot;k){\bf L}_{q,\omega}^k{\bf u}\quad(k=0,1,\ldots)\;,
\end{equation}
where $\Phi(\cdot;k)$ is the polynomial given by (\ref{Rod-knPhi}).
Moreover, ${\bf u}^{[k]}$ fulfils the functional equation (see \cite[Theorem 3.1]{F1998})
\begin{equation}\label{feq-unk}
{\bf D}_{q,\omega}\big(\phi{\bf u}^{[k]}\big)=\psi^{[k]}{\bf u}^{[k]}\quad(k=0,1,\ldots)\;,
\end{equation}
where $\psi^{[k]}\in\mathcal{P}_1$ is defined by
\begin{equation}\label{def-psik}
\psi^{[0]}:=\psi\;,\quad
\psi^{[k]}:=D_{q,\omega}\phi+qL_{q,\omega}\psi^{[k-1]}\quad(k=1,2,\ldots)\;.
\end{equation}
We point out that equality (\ref{feq-unk}) was stated in \cite{F1998}
under the assumption that ${\bf u}$ is a regular functional, but inspection of the proof given therein shows that
the equality remains true without such assumption.
Using mathematical induction on $k$, we prove that $\psi^{[k]}$ is explicitly given by
\begin{equation}\label{explicit-psik}
\psi^{[k]}(x)=d_{2k}x+e_k\quad(k=0,1,\ldots)\;,
\end{equation}
where $d_{2k}$ and $e_k$ are defined by (\ref{abcpr2-Dqw}).
This representation (\ref{explicit-psik}) has not been observed in \cite{F1998}.
It will play a central role along this work.

\begin{definition}\label{qw-admissible-pair}
$(\phi,\psi)$ is called a $(q,\omega)-$admissible pair if
\begin{equation}\label{qw-admiss-cond}
\phi\in\mathcal{P}_2\;,\quad\psi\in\mathcal{P}_1\;,\quad
d_n:=\psi'\,q^n+\mbox{$\frac12$}\phi^{''}\,[n]_{q}\neq0\quad\big(\forall n\in\mathbb{N}_0\big)\;.
\end{equation}
\end{definition}

\begin{lemma}\label{lemmaRn+1}
Let ${\bf u}\in\mathcal{P}^*\setminus\{{\bm 0}\}$.
Suppose that ${\bf u}$ satisfies $(\ref{2.1})$,
where $\phi\in\mathcal{P}_2$ and $\psi\in\mathcal{P}_1$.
Let $(Q_n)_{n\geq0}$ be any simple set of polynomials and define
\begin{equation}\label{def-Rn+1}
\begin{array}{rl}
R_{n+1}(x)& :=\phi(x) D_{q,\omega}^* Q_n(x) +q \psi(x) Q_n (x) \\ [0.5em]
& \;=a_nq^{1-n}d_nx^{n+1}+\mbox{\rm (lower degree terms)}\;,
\end{array}
\end{equation}
where $a_n\in\mathbb{C}\setminus\{0\}$ is the leading coefficient
of $Q_n$ and $d_n$ is defined as in $(\ref{qw-admiss-cond})$.
Then the following functional equation holds:
\begin{equation}\label{Dqw-RQ}
{\bf D}_{q,\omega}^*\big(Q_{n} {\bf u}^{[1]}\big)=R_{n+1} {\bf u}\quad(n=0,1,\ldots)\;.
\end{equation}
Moreover, $(R_n)_{n\geq0}$ is a simple set of polynomials
if and only if $(\phi,\psi)$ is a $(q,\omega)-$admissible pair,
provided that we define $R_0(x):=1$. 
\end{lemma}

\begin{proof}
Fix $n\in\mathbb{N}_0$ and take arbitrarily $f\in\mathcal{P}$. Then
\begin{align*}
\big\langle  {\bf D}_{q,\omega}^* \big(Q_n {\bf u}^{[1]}\big), f \big\rangle
& =-q  \langle  {\bf L}_{q,\omega} (\phi {\bf u}), Q_n D_{q,\omega} f \rangle
=-\big\langle  \phi {\bf u}, \big(L_{q,\omega}^* Q_n\big)\big(L_{q,\omega}^* D_{q,\omega} f\big) \big\rangle \\
& =-\big\langle  \phi {\bf u}, \big(L_{q,\omega}^* Q_n\big)\big(D_{q,\omega}^* f\big) \big\rangle\;.
\end{align*}
Now, using relation (\ref{P5}) with $D_{q,\omega}^*$ instead of $D_{q,\omega}$, we obtain
\begin{align*}
\big\langle {\bf D}_{q,\omega}^* \big(Q_n {\bf u}^{[1]}\big), f \big\rangle
& =-\langle\phi {\bf u}, D_{q,\omega}^*(fQ_n)-fD_{q,\omega}^*Q_n \rangle \\
& =q\big\langle  {\bf D}_{q,\omega}(\phi {\bf u}), Q_nf\rangle+\langle\phi {\bf D}_{q,\omega}^*Q_n, f\rangle \\
& =\big\langle  {\bf u}, \big(q\psi Q_n+\phi D_{q,\omega}^*Q_n\big) f\big\rangle
=\langle R_{n+1}{\bf u}, f\rangle\;.
\end{align*}
This proves (\ref{Dqw-RQ}).
Moreover, taking into account (\ref{Dqwxn}), we have
\begin{align*}
D_{q,\omega}^*Q_n(x)&=a_nD_{1/q,-\omega/q}\,x^n+\mbox{\rm (lower degree terms)} \\
&=a_nq^{1-n}[n]_q\,x^{n-1}+\mbox{\rm (lower degree terms)}\,,
\end{align*}
where we also took into account that $[n]_{q^{-1}}=q^{1-n}[n]_q$. Hence
$$
R_{n+1}(x) =\big(aa_nq^{1-n}[n]_q+qa_nd\big)x^{n+1}+\mbox{\rm (lower degree terms)\;,}
$$
and so we obtain the expression for $R_{n+1}$ given in (\ref{def-Rn+1}).
Thus, $\deg R_{n+1}=n+1$ for each $n=0,1,\ldots$ if and only if $d_n\neq0$ for each $n=0,1,\ldots$, i.e.,
if and only if $(\phi,\psi)$ is a $(q,\omega)-$admissible pair. This concludes the proof.
\end{proof}

In the statement of the next lemma, which is interesting for its own sake,
we emphasize that neither the given functional ${\bf u}$ needs to be regular
nor the sequence $(P_n)_{n\geq0}$ needs to be an OPS.
Under the assumption that ${\bf u}$ is regular and satisfies (\ref{2.1}), formula (\ref{RodRn})
in bellow may be derived in a very simple way (see Remark \ref{Rmk7} below).

\begin{lemma}\label{lemmaRodFunctional}
Let ${\bf u}\in\mathcal{P}^*\setminus\{{\bm 0}\}$.
Suppose that ${\bf u}$ satisfies $(\ref{2.1})$,
where $(\phi,\psi)$ is a $(q,\omega)-$admissible pair.
Then the Rodrigues-type formula
\begin{equation}\label{RodRn}
P_n{\bf u}=k_n{\bf D}_{1/q,-\omega/q}^n\,{\bf u}^{[n]}\quad(n=0,1,\ldots)
\end{equation}
holds in $\mathcal{P}^*$, where $k_n$ is defined as in $(\ref{Rod-knPhi})$
and $(P_n)_{n\geq0}$ is a simple set of monic polynomials given by the three-term recurrence relation
$(\ref{ttrrC1-Dqw})$--$(\ref{EqDistC3-Dqw2})$.
\end{lemma}

\begin{proof}
Since $(\phi,\psi)$ is a $(q,\omega)-$admissible pair, then $d_n\neq0$ for each $n=0,1,\ldots$.
Hence the sequence $(P_n)_{n\geq0}$ given by $(\ref{ttrrC1-Dqw})$--$(\ref{EqDistC3-Dqw2})$ is well defined.
For simplicity, we set ${\bf H}_{q,\omega}:={\bf D}_{q,\omega}^*:={\bf D}_{1/q,-\omega/q}$,
and so (\ref{RodRn}) reads as
\begin{equation}\label{RodRnH}
P_n{\bf u}=k_n{\bf H}_{q,\omega}^n\,{\bf u}^{[n]}\quad(n=0,1,\ldots)\;.
\end{equation}
Notice that the second relation in (\ref{P4}) can be rewritten as
\begin{equation}\label{Rodqw4}
{\bf H}_{q,\omega}\,{\bf L}_{q,\omega}=q{\bf D}_{q,\omega}\,,
\end{equation}
while, setting $H_{q,\omega}:=D_{q,\omega}^*$,
Leibnitz rule (\ref{Leibnitz-Dqwnfu}) applied to ${\bf D}_{q,\omega}^*$ gives
\begin{equation}\label{Rodqw5}
{\bf H}_{q,\omega}^n(f{\bf u})=\sum_{k=0}^n \qbin{n}{k}{q^{-1}}
L_{q,\omega}^{*k}\big(H_{q,\omega}^{n-k}f\big)\, {\bf H}_{q,\omega}^k{\bf u}
\quad(f\in\mathcal{P})\;.
\end{equation}
We will prove (\ref{RodRnH}) by mathematical induction on $n$.
For $n=0$, (\ref{RodRnH}) becomes a trivial equality.
For $n=1$, we use (\ref{def-unk}) and (\ref{Rodqw4}) to deduce
$$
{\bf H}_{q,\omega}\,{\bf u}^{[1]}
={\bf H}_{q,\omega} {\bf L}_{q,\omega}\big(\phi{\bf u}\big)
=q{\bf D}_{q,\omega}\big(\phi{\bf u}\big)=q\psi{\bf u} \;.
$$
Therefore, since $P_1(x)=x-\beta_0=x-(-e_0/d_0)=x+e/d=d^{-1}\psi(x)$, and so $q\psi=qdP_1=k_1^{-1}P_1$,
we obtain (\ref{RodRnH}) for $n=1$.
Assume now that (\ref{RodRnH}) holds for given consecutive numbers $n-1$ and $n$
($n\in\mathbb{N}$), i.e., suppose that (induction hypothesis)
\begin{equation}
P_{n-1}{\bf u}=k_{n-1}{\bf H}_{q,\omega}^{n-1}\,{\bf u}^{[n-1]}\;,\quad
P_n{\bf u}=k_n{\bf H}_{q,\omega}^n\,{\bf u}^{[n]} \,. \label{Rodqw1ab}
\end{equation}
We need to show that
\begin{equation}\label{Rodqw3}
P_{n+1}{\bf u}=k_{n+1}{\bf H}_{q,\omega}^{n+1}\,{\bf u}^{[n+1]}\,.
\end{equation}
To prove (\ref{Rodqw3}), we start by noting that
\begin{equation}\label{Rodqw6}
{\bf H}_{q,\omega}^{n+1}\,{\bf u}^{[n+1]}=q{\bf H}_{q,\omega}^{n}\big(\psi^{[n]}{\bf u}^{[n]}\big)\;.
\end{equation}
Indeed, using successively (\ref{def-unk}) and (\ref{Rodqw4}), we have
$$
{\bf H}_{q,\omega}^{n+1}\,{\bf u}^{[n+1]}
={\bf H}_{q,\omega}^{n}\Big({\bf H}_{q,\omega}{\bf L}_{q,\omega}(\phi{\bf u}^{[n]}\big)\Big)
=q{\bf H}_{q,\omega}^{n}\Big({\bf D}_{q,\omega}\big(\phi{\bf u}^{[n]}\big)\Big)\;,
$$
and so (\ref{Rodqw6}) follows taking into account (\ref{feq-unk}).
Next, by (\ref{Rodqw5}) with $f=\psi^{[n]}=d_{2n}x+e_n$,
$$
{\bf H}_{q,\omega}^{n}\big(\psi^{[n]}{\bf u}^{[n]}\big)
=\big(L_{q,\omega}^{*n}\,\psi^{[n]}\big){\bf H}_{q,\omega}^n{\bf u}^{[n]}
+[n]_{q^{-1}}d_{2n}{\bf H}_{q,\omega}^{n-1}{\bf u}^{[n]}\;.
$$
Replacing this into (\ref{Rodqw6})
and using the second identity in (\ref{Rodqw1ab}), we deduce
\begin{equation}\label{Rodqw7}
[n]_{q^{-1}}d_{2n}{\bf H}_{q,\omega}^{n-1}{\bf u}^{[n]}
=q^{-1}{\bf H}_{q,\omega}^{n+1}\,{\bf u}^{[n+1]}
-k_n^{-1}\big(L_{q,\omega}^{*n}\,\psi^{[n]}\big)P_n{\bf u}\;.
\end{equation}
Taking into account both identities appearing in (\ref{Rodqw1ab}),
we may change $n$ into $n-1$ in the preceding reasoning, to obtain
\begin{equation}\label{Rodqw7b}
[n-1]_{q^{-1}}d_{2n-2}{\bf H}_{q,\omega}^{n-2}{\bf u}^{[n-1]}
=\Big(q^{-1}k_n^{-1}P_n-k_{n-1}^{-1}\big(L_{q,\omega}^{*n-1}\,\psi^{[n-1]}\big)P_{n-1}\Big){\bf u}\;.
\end{equation}
Next, by the analogue of (\ref{P6}) for ${\bf D}_{q,\omega}^*$, we have
\begin{align}
{\bf H}_{q,\omega}\big(\psi^{[n]}{\bf u}^{[n]}\big)
&=\big(D_{q,\omega}^*\psi^{[n]}\big){\bf L}_{q,\omega}^*{\bf u}^{[n]}+\psi^{[n]}{\bf H}_{q,\omega}{\bf u}^{[n]} \nonumber \\
&= d_{2n}{\bf L}_{q,\omega}^*{\bf L}_{q,\omega}\big(\phi{\bf u}^{[n-1]}\big)
+\psi^{[n]}{\bf H}_{q,\omega}{\bf L}_{q,\omega}\big(\phi{\bf u}^{[n-1]}\big) \nonumber \\
&= \big(d_{2n}\phi+q\psi^{[n]}\psi^{[n-1]}\big){\bf u}^{[n-1]}\;, \label{Rodqw8}
\end{align}
where in the last equality we used (\ref{P2}), (\ref{Rodqw4}), and (\ref{feq-unk}).
From (\ref{Rodqw6}) and (\ref{Rodqw8}), we obtain
\begin{equation}\label{Rodqw9}
{\bf H}_{q,\omega}^{n+1}\,{\bf u}^{[n+1]}
=q{\bf H}_{q,\omega}^{n-1}\big(\theta_2(\cdot;n){\bf u}^{[n-1]}\big)\;,
\end{equation}
where $\theta_2(x;n):=d_{2n}\phi+q\psi^{[n]}\psi^{[n-1]}$.
Since $\deg\theta_2(\cdot;n)\leq2$, applying the Leibnitz formula (\ref{Rodqw5}) to the right-hand side of
(\ref{Rodqw9}), we obtain
\begin{align}
{\bf H}_{q,\omega}^{n+1}\,{\bf u}^{[n+1]}
&=qL_{q,\omega}^{*\,n-1}\big(\theta_2(\cdot;n)\big){\bf H}_{q,\omega}^{n-1}{\bf u}^{[n-1]} \nonumber \\
&\quad+q[n-1]_{q^{-1}}L_{q,\omega}^{*\,n-2}\big(D_{q,\omega}^*
\theta_2(\cdot;n)\big){\bf H}_{q,\omega}^{n-2}{\bf u}^{[n-1]} \label{Rodqw10} \\
&\quad+\frac{q[n-1]_{q^{-1}}[n-2]_{q^{-1}}}{[2]_{q^{-1}}}L_{q,\omega}^{*\,n-3}\big(D_{q,\omega}^{*\,2}
\theta_2(\cdot;n)\big){\bf H}_{q,\omega}^{n-3}{\bf u}^{[n-1]}  \nonumber\;.
\end{align}
Now, since $\phi(x)=ax^2+bx+c$, $\psi^{[k]}=d_{2k}x+e_k$, and the relations
$$
d_{k+1}=a+qd_k\;,\quad
e_{k+1}=b+qe_k+\omega d_{2k+1}\;,\quad
d_{2k+2}+qd_{2k}=(1+q)d_{2k+1}
$$
hold for each $k=0,1,\ldots$, we show that $\theta_2(\cdot;n)$ is given explicitly by
$$
\theta_2(x;n)=d_{2n}d_{2n-1}x^2+d_{2n-1}\big((1+q)e_n-\omega d_{2n}\big)x+cd_{2n}+qe_ne_{n-1}\;.
$$
(Hence, $\deg\theta_2(\cdot;n)=2$.) From this and taking into account (\ref{Dqwxn}), we compute
\begin{align*}
&D_{q,\omega}^*\big(\theta_2(\cdot;n)\big)
=[2]_{q^{-1}}d_{2n-1}\big(d_{2n}x+qe_n-\omega d_{2n}\big)\;, \\
&D_{q,\omega}^{*\,2}\big(\theta_2(\cdot;n)\big)=[2]_{q^{-1}}d_{2n-1}d_{2n}\;.
\end{align*}
Moreover, by (\ref{P1}),
$$
L_{q,\omega}^{*k}1=1\;,\quad
L_{q,\omega}^{*k}x=q^{-k}\big(x-\omega[k]_q\big)\;,\quad
L_{q,\omega}^{*k}x^2=q^{-2k}\big(x^2-2\omega[k]_qx+\omega^2[k]_q^2\big)
$$
for each $k=0,1,\ldots$, hence we deduce
\begin{align}
L_{q,\omega}^{*\,n-1}\big(\theta_2(\cdot;n)\big)&=q^{2-2n}d_{2n}d_{2n-1}x^2 \nonumber \\
&\quad+q^{1-n}d_{2n-1}\Big((1+q)e_n-\omega d_{2n}\big([n]_{q^{-1}}+q^{-1}[n-1]_{q^{-1}}\big)\Big)x \nonumber \\
&\quad+\omega^2q^{1-n}d_{2n}d_{2n-1}[n-1]_{q}[n]_{q^{-1}} \nonumber \\
&\quad\quad+qe_n\big(e_{n-1}-\omega(1+q)d_{2n-1}q^{-n}[n-1]_q\big)+cd_{2n}\;,  \label{Rodqw9aa} \\
\rule{0pt}{1.2em}
L_{q,\omega}^{*\,n-2}\big(D_{q,\omega}^*\theta_2(\cdot;n)\big)
&=[2]_{q^{-1}}d_{2n-1}\big(d_{2n}q^{2-n}x+qe_n-\omega d_{2n}[n-1]_{q^{-1}}\big)\;, \label{Rodqw9ab} \\
\rule{0pt}{1.2em}
L_{q,\omega}^{*\,n-3}\big(D_{q,\omega}^{*\,2}\theta_2(\cdot;n)\big)
&=[2]_{q^{-1}}d_{2n-1}d_{2n}\;. \label{Rodqw9ac}
\end{align}
Relation (\ref{Rodqw9ac}) allow us to rewrite (\ref{Rodqw10}) as
\begin{align}
&q[n-1]_{q^{-1}}[n-2]_{q^{-1}}d_{2n-1}d_{2n}{\bf H}_{q,\omega}^{n-3}{\bf u}^{[n-1]}  \nonumber\\
&\hspace{5em}= {\bf H}_{q,\omega}^{n+1}\,{\bf u}^{[n+1]}
-qL_{q,\omega}^{*\,n-1}\big(\theta_2(\cdot;n)\big){\bf H}_{q,\omega}^{n-1}{\bf u}^{[n-1]} \label{Rodqw10aa}\\
&\hspace{5em}\quad-q[n-1]_{q^{-1}}L_{q,\omega}^{*\,n-2}\big(D_{q,\omega}^*
\theta_2(\cdot;n)\big){\bf H}_{q,\omega}^{n-2}{\bf u}^{[n-1]}\;. \nonumber
\end{align}
On the other hand,
\begin{align*}
{\bf H}_{q,\omega}^{n-1}\,{\bf u}^{[n]} &
={\bf H}_{q,\omega}^{n-2}\big({\bf H}_{q,\omega}{\bf L}_{q,\omega}\big(\phi{\bf u}^{[n-1]}\big)
=q{\bf H}_{q,\omega}^{n-2}\big({\bf D}_{q,\omega}\big(\phi{\bf u}^{[n-1]}\big)
=q{\bf H}_{q,\omega}^{n-2}\big(\psi^{[n-1]}{\bf u}^{[n-1]}\big) \\
&=qL_{q,\omega}^{*\,n-2}\big(\psi^{[n-1]}\big){\bf H}_{q,\omega}^{n-2}{\bf u}^{[n-1]}
+q[n-2]_{q^{-1}}L_{q,\omega}^{*\,n-3}\big(D_{q,\omega}^*\psi^{[n-1]})\big){\bf H}_{q,\omega}^{n-3}{\bf u}^{[n-1]} \;,
\end{align*}
where in the last equality we used once again the Leibnitz formula.
As a consequence, since $L_{q,\omega}^{*\,n-3}\big(D_{q,\omega}^*\psi^{[n-1]})\big)=d_{2n-2}$, we obtain
\begin{equation}\label{Rodqw11}
q[n-2]_{q^{-1}}d_{2n-2}{\bf H}_{q,\omega}^{n-3}{\bf u}^{[n-1]}
={\bf H}_{q,\omega}^{n-1}\,{\bf u}^{[n]}-qL_{q,\omega}^{*\,n-2}\big(\psi^{[n-1]}\big){\bf H}_{q,\omega}^{n-2}{\bf u}^{[n-1]}\;.
\end{equation}
Substituting in (\ref{Rodqw10aa}) the expression for
${\bf H}_{q,\omega}^{n-3}{\bf u}^{[n-1]}$ given by (\ref{Rodqw11}),
and then taking into account (\ref{Rodqw7}) and (\ref{Rodqw7b}),
as well as the first equation in (\ref{Rodqw1ab}), we deduce
\begin{align}
\left(1-\frac{q^{-1}[n-1]_{q^{-1}}}{[n]_{q^{-1}}}\frac{d_{2n-1}}{d_{2n-2}}\right)
{\bf H}_{q,\omega}^{n+1}\,{\bf u}^{[n+1]}
=\Big(A(\cdot;n)P_n+B(\cdot;n)P_{n-1}\Big){\bf u}\;, \label{AxnBxn}
\end{align}
where $A(\cdot;n)$ and $B(\cdot;n)$ are polynomials given by
\begin{align*}
A(x;n)&:=\left.\frac{k_n^{-1}}{d_{2n-2}}\right\{
-\frac{[n-1]_{q^{-1}}d_{2n-1}\big(L_{q,\omega}^{*\,n}\psi^{[n]}\big)}{[n]_{q^{-1}}} \\
&\quad\qquad\qquad\left.
+L_{q,\omega}^{*\,n-2}\big(D_{q,\omega}^*\theta_2(\cdot;n)\big)
-\frac{d_{2n-1}d_{2n}}{d_{2n-2}}\big(L_{q,\omega}^{*\,n-2}\psi^{[n-1]}\big)\right\}
\end{align*}
and
\begin{align*}
B(x;n)&:=\left.\frac{qk_{n-1}^{-1}}{d_{2n-2}}\right\{
d_{2n-2}L_{q,\omega}^{*\,n-1}\big(\theta_2(\cdot;n)\big)
-\big(L_{q,\omega}^{*\,n-1}\psi^{[n-1]}\big)\times \\
&\quad\qquad\qquad\left. \rule{0pt}{1.7em}
\times\left(L_{q,\omega}^{*\,n-2}\big(D_{q,\omega}^*\theta_2(\cdot;n)\big)
-\frac{d_{2n-1}d_{2n}}{d_{2n-2}}\big(L_{q,\omega}^{*\,n-2}\psi^{[n-1]}\big)\right)\right\}\,.
\end{align*}
Now, taking into account (\ref{Rodqw9aa}) and (\ref{Rodqw9ab}), as well as the relations
\begin{align*}
L_{q,\omega}^{*\,n}\psi^{[n]}(x)&=q^{-n}d_{2n}x+e_n-\omega[n]_qq^{-n}d_{2n} \;,\\
L_{q,\omega}^{*\,n-2}\psi^{[n-1]}(x)&=q^{2-n}d_{2n-2}x+e_{n-1}-\omega[n-2]_qq^{2-n}d_{2n-2} \;,
\end{align*}
and also making use of the identities
$$
k_n^{-1}=\frac{q^{n-1}d_{n-1}}{d_{2n}d_{2n-1}}k_{n+1}^{-1}\;,\quad
k_{n-1}^{-1}=\frac{q^{2n-3}d_{n-1}d_{n-2}}{d_{2n}d_{2n-1}d_{2n-2}d_{2n-3}}k_{n+1}^{-1}\;,
$$
it is straightforward to verify that 
$$
A(x;n)=k_{n+1}^{-1}\frac{q^{n-1}d_{n-1}}{[n]_qd_{2n-2}}\big(x-\beta_n\big)\;,\quad
B(x;n)=-k_{n+1}^{-1}\frac{q^{n-1}d_{n-1}}{[n]_qd_{2n-2}}\gamma_{n} \;,
$$
$\beta_n$ and $\gamma_{n}$ being given by (\ref{EqDistC3-Dqw1})--(\ref{EqDistC3-Dqw2}).
Finally, replacing these expressions for $A(\cdot;n)$ and $B(\cdot;n)$ in the right-hand side of (\ref{AxnBxn}),
and taking into account (\ref{ttrrC1-Dqw}) and the identity
$$
1-\frac{q^{-1}[n-1]_{q^{-1}}}{[n]_{q^{-1}}}\frac{d_{2n-1}}{d_{2n-2}}=
\frac{q^{n-1}d_{n-1}}{[n]_qd_{2n-2}}\;,
$$
(\ref{Rodqw3}) follows.
\end{proof}

Lemma \ref{lemma-phipsiNOTzero} in bellow can be easily proved (see \cite[Lemma 3.1]{F1998}).

\begin{lemma}\label{lemma-phipsiNOTzero}
Let ${\bf u}\in\mathcal{P}^*$. Suppose that ${\bf u}$ is regular and fulfills $(\ref{2.1})$,
with $\phi\in\mathcal{P}_2$ and $\psi\in\mathcal{P}_1$.
If at least one of the polynomials $\phi$ and $\psi$ is not the zero polynomial,
then none of these polynomials can be the zero polynomial and, moreover, $\deg\psi=1$.
\end{lemma}

The statement of the next lemma is given in \cite[Lemma 3.5]{F1998}.
We highlight that the proof of the $(q,\omega)-$admissibility condition is incorrect (see \cite[Lemma 3.5--(i)]{F1998}),
and so the proof therein may be regarded as incomplete.
For sake of completeness, we present a proof following the ideas presented in \cite{MP1994}.

\begin{lemma}\label{lemma-admissible}
Let ${\bf u}\in\mathcal{P}^*$. Suppose that ${\bf u}$ is regular
and satisfies $(\ref{2.1})$, where $\phi\in\mathcal{P}_2$,
$\psi\in\mathcal{P}_1$, and at least one of the polynomials $\phi$ and $\psi$ is not the zero polynomial.
Then $(\phi,\psi)$ is a $(q,\omega)-$admissible pair and ${\bf u}^{[k]}$ is regular for each $k\in\mathbb{N}$.
Moreover, if $(P_n)_{n\geq0}$ is the monic OPS with respect to ${\bf u}$, then
$\big(P_n^{[k]}\big)_{n\geq0}$ is the monic OPS with respect to ${\bf u}^{[k]}$.
\end{lemma}

\begin{proof}
We start by considering the case $k=1$.
Set $Q_n:=P_n^{[1]}=D_{q,\omega}P_{n+1}/[n+1]_q$ and
let $R_{n+1}$ be the corresponding polynomial defined by (\ref{def-Rn+1}).
Fix arbitrarily $m,n\in\mathbb{N}_0$, with $m\leq n$.
Then, by Lemma \ref{lemmaRn+1},
\begin{align*}
[m+1]_q\big\langle {\bf u}^{[1]}, Q_nQ_m \big\rangle
& =-\big\langle  {\bm D}_{q,\omega}^*\big(Q_n {\bf u}^{[1]}\big), P_{m+1}\big\rangle
=-q^{-1}\langle R_{n+1}{\bf u}, P_{m+1}\rangle \\
& =-q^{-n}d_n\langle {\bf u}, P_{n+1}^2\rangle\delta_{m,n}\;,
\end{align*}
hence we obtain
\begin{equation}\label{ortogDqwPn}
\big\langle {\bf u}^{[1]}, P_n^{[1]}P_m^{[1]} \big\rangle=-\frac{q^{-n} d_n}{[n+1]_q} \langle {\bf u}, P_{n+1}^2\rangle \delta_{m,n}\quad(m,n=0,1,\ldots)\;.
\end{equation}
Next, let $s:=\deg\phi\in\{0,1,2\}$. Then
\begin{equation}\label{qqq1}
0\neq\big\langle {\bf u}, \phi\big(L_{q,\omega}^*P_n^{[1]}\big)P_{n+s}\big\rangle
=\big\langle \phi{\bf u},L_{q,\omega}^* \big(P_n^{[1]}L_{q,\omega}P_{n+s}\big)\big\rangle
= q\big\langle {\bf u}^{[1]},P_n^{[1]}L_{q,\omega}P_{n+s} \rangle \;.
\end{equation}
Since $L_{q,\omega}P_{n+s}(x)=\sum_{m=0}^{n+s}c_{n,m}P_m^{[1]}(x)$ for some coefficients $c_{n,m}\equiv c_{n,m}(s;q,\omega)\in\mathbb{C}$, from (\ref{ortogDqwPn}) and (\ref{qqq1}) we deduce
\begin{equation}\label{qqq2}
0\neq\sum_{m=0}^{n+s}c_{n,m} \big\langle {\bf u}^{[1]}, P_n^{[1]}P_m^{[1]} \big\rangle=-\frac{q^{-n} d_n c_{n,n}}{[n+1]_q} \langle {\bf u}, P_{n+1}^2\rangle \quad(n=0,1,\ldots)\;.
\end{equation}
This implies $d_n\neq0$ (and also $c_{n,n}\neq0$) for each $n=0,1,\ldots$, which means that $(\phi,\psi)$ is a $(q,\omega)-$admissible pair. Thus, it follows from (\ref{ortogDqwPn}) that
$\big(P_n^{[1]}\big)_{n\geq0}$ is a monic OPS with respect to ${\bf u}^{[1]}$.
This proves the last statement in the theorem for $k=1$.
Now, by (\ref{feq-unk}), ${\bf u}^{[1]}$ fulfills
${\bf D}_{q,\omega}\big(\phi{\bf u}^{[1]}\big)=\psi^{[1]}{\bf u}^{[1]}$, hence, since
$P_n^{[2]}=D_{q,\omega}P_{n+1}^{[1]}/[n+1]_q$ and, by (\ref{explicit-psik}), $\psi^{[1]}(x)=d_2x+e_1$,
from (\ref{ortogDqwPn}) with ${\bf u}$, $\psi$, and $(P_n)_{n\geq0}$
replaced (respectively) by ${\bf u}^{[1]}$, $\psi^{[1]}$, and
$(P_n^{[1]})_{n\geq0}$, we deduce, for every $n,m\in\mathbb{N}_0$,
$$
\langle {\bf u}^{[2]} , P_n^{[2]} P_m^{[2]} \rangle
=-\frac{q^{-n}d_n^{[1]}}{[n+1]_q} \langle {\bf u}^{[1]} , \big(P_{n+1}^{[1]}\big)^2 \rangle \delta_{nm}
$$
where $d_n^{[1]}$ is defined as in (\ref{qw-admiss-cond})
corresponding to the pair $(\phi,\psi^{[1]})$, so that
$$
d_n^{[1]}:=\big(\psi^{[1]}\big)'\,q^n+\mbox{$\frac12$}\phi^{''}\,[n]_{q}=d_2q^n+a[n]_q=d_{n+2}\;.
$$
Therefore, and taking into account once again (\ref{ortogDqwPn}), we obtain
$$
\langle{\bf u}^{[2]},P_n^{[2]}P_m^{[2]}\rangle
=q^{-(2n+1)}\frac{d_{n+1}d_{n+2}}{[n+1]_q[n+2]_q}\langle{\bf u},P_{n+2}^2\rangle\delta_{nm}\quad(n,m\in\mathbb{N}_0)\;,
$$
and so $\{P_n^{[2]}\}_{n\geq0}$ is a monic OPS with respect to ${\bf u}^{[2]}$.
Arguing by induction, we prove
\begin{equation}\label{NuPaa}
\langle {\bf u}^{[k]} , P_n^{[k]} P_m^{[k]} \rangle
=(-1)^kq^{-k(2n+k-1)/2}\Big(\prod_{j=1}^{k}\frac{d_{n+k+j-2}}{[n+j]_q}\Big)\langle{\bf u},P_{n+k}^2\rangle\delta_{nm}
\quad (k,n,m\in\mathbb{N}_0)\;,
\end{equation}
hence $\{P_n^{[k]}\}_{n\geq0}$ is a monic OPS with respect to
${\bf u}^{[k]}$, for each $k\in\mathbb{N}_0$.
\end{proof}

\subsection{Proof of Theorem \ref{Dqw-main-Thm}}

Suppose that ${\bf u}$ is regular. Fix $n\in\mathbb{N}_0$.
Since ${\bf u}$ satisfies (\ref{EqDistC1-Dqw}),
Lemma \ref{lemma-admissible} ensures that
$(\phi,\psi)$ is a $(q,\omega)-$admissible pair, and so $d_n\neq0$.
Moreover, ${\bf u}^{[n]}$ is regular and $\big(P_j^{[n]}\big)_{j\geq0}$
is the corresponding monic OPS, which fulfills a three-term recurrence relation: 
\begin{equation}\label{ttrrC1}
P_{j+1}^{[n]}(x)=(x-\beta_j^{[n]})P_j^{[n]}(x)-\gamma_j^{[n]}P_{j-1}^{[n]}(x)\quad  (j=0,1,\ldots)\;,
\end{equation}
where $P_{-1}^{[n]}(x)=0$, being $\beta_{j}^{[n]}\in\mathbb{C}$
and $\gamma_{j}^{[n]}\in\mathbb{C}\setminus\{0\}$ for each $j$.
Let us compute $\gamma_{1}^{[n]}$. 
We first show that (for $n=0$) the coefficient $\gamma_1\equiv\gamma_{1}^{[0]}$,
appearing in the three-term recurrence relation for $\{P_j\}_{j\geq0}$,
is given by
\begin{equation}\label{g1}
\gamma_1=-\frac{1}{dq+a}\,\phi\left(-\frac{e}{d}\right) \; .
\end{equation}
This may be proved taking  $n=0$ and $n=1$ in the relation
$\langle{\bf D}_{q,\omega}(\phi{\bf u}),x^n\rangle=\langle\psi{\bf u},x^n\rangle$.
Indeed, setting $u_n:=\langle{\bf u},x^n\rangle$, for $n=0$ we obtain $0=du_1+eu_0$, and for $n=1$ we find
$-q^{-1}(au_2+bu_1+cu_0)=du_2+eu_1$. Therefore,
\begin{equation}
u_1=-\frac{e}{d}u_0 \; , \quad u_2=-\frac{1}{dq+a}\left[-(qe+b)\frac{e}{d}+c \right] u_0 \; .\label{u1}
\end{equation}
On the other hand, since $P_1(x)=x-\beta_0=x-u_1/u_0$, we also have
\begin{equation}
\gamma_1=\frac{\langle{\bf u},P_1^2\rangle}{u_0}=\frac{u_2u_0-u_1^2}{u_0^2}
=\frac{u_2}{u_0}-\left(\frac{u_1}{u_0}\right)^2 \; .
\label{u2}
\end{equation}
Substituting $u_1$ and $u_2$ given by (\ref{u1}) into (\ref{u2}) yields (\ref{g1}).
Now, since equation (\ref{feq-unk}) is of the same type as
(\ref{EqDistC1-Dqw}), with the same polynomial $\phi$ and being $\psi$ replaced by $\psi^{[n]}$,
we see that $\gamma_{1}^{[n]}$ may be obtained
replacing in (\ref{g1}) the coefficients $d$ and $e$ of $\psi(x)=dx+e$
by the corresponding coefficients of $\psi^{[n]}(x)=d_{2n}x+e_n$. Hence,
\begin{equation}\label{g1n}
\gamma_{1}^{[n]}=-\frac{1}{d_{2n}q+a}\phi\left(-\frac{e_n}{d_{2n}}\right)
=-\frac{1}{d_{2n+1}}\phi\left(-\frac{e_n}{d_{2n}}\right) \;.
\end{equation}
Since ${\bf u}^{[n]}$ is regular, then $\gamma_{1}^{[n]}\neq0$,
hence $\phi\left(-\frac{e_n}{d_{2n}}\right)\neq0$.
Thus, (\ref{EqDistC2-Dqw}) holds.

Conversely, suppose that (\ref{EqDistC2-Dqw}) holds.
Then, by Favard's theorem,
the sequence $(P_n)_{n\geq0}$ defined by the
three-term recurrence relation (\ref{ttrrC1-Dqw})--$(\ref{EqDistC3-Dqw2})$ is a monic OPS.
We claim that $\{P_n\}_{n\geq0}$ is an OPS with respect to ${\bf u}$.
To prove this sentence we only need to show that
(see e.g. \cite[Chapter I, Exercise 4.14]{C1978} or \cite[Corollary 6.2]{P2018})
\begin{equation}\label{qq}
\langle{\bf u},1\rangle\neq0 \;,\quad \langle{\bf u},P_n\rangle=0 \quad (n=1,2,\ldots)\;.
\end{equation}
Suppose that $\langle{\bf u},1\rangle=0$.
Since the functional equation (\ref{EqDistC1-Dqw}) is equivalent to
the second order difference equation (\ref{uDqwxn})
fulfilled by the moments $y_n:=\langle{\bf u},Y_n\rangle$, 
and noting that for $n=0$ (\ref{uDqwxn}) yields $dy_1+ey_0=0$,
we get $y_1=0$ (because $y_0=\langle{\bf u},1\rangle=0$ and $d=d_0\neq0$);
hence $y_0=y_1=0$ and so it follows recurrently from (\ref{uDqwxn})
that $y_n=0$ for each $n\in\mathbb{N}_0$. Therefore ${\bf u}=\textbf{0}$,
in contradiction with the hypothesis. Thus, $\langle{\bf u},1\rangle\neq0$.
On the other hand, by Lemma \ref{lemmaRodFunctional}, for each $n\geq1$ we may write
$$\langle{\bf u},P_n\rangle=\langle P_n{\bf u},1\rangle
=-qk_n\big\langle{\bf D}_{1/q,-\omega/q}^{n-1}{\bf u}^{[n]},D_{q,\omega}1\big\rangle
=0\;.$$
Thus (\ref{qq}) is proved,
hence ${\bf u}$ is regular and $(P_n)_{n\geq0}$ is the corresponding monic OPS.
Finally, the Rodrigues-type formula (\ref{EqDistRod-Dqw})
follows from Lemma \ref{lemmaRodFunctional} and (\ref{repres-unk}),
concluding the proof of Theorem \ref{Dqw-main-Thm}.

\begin{remark}
Since $-e_n/d_{2n}$ is the unique zero of $\psi^{[n]}(x)=d_{2n}x+e_n$,
the regularity conditions (\ref{EqDistC2-Dqw}) for ${\bf u}$ given in Theorem \ref{Dqw-main-Thm}
may be restated as follows:
{\it ${\bf u}$ is regular if and only if
$(\phi,\psi)$ is a $(q,\omega)-$admissibe pair and $\psi^{[n]}\nmid\phi$ for each $n=0,1,\ldots$.}
Thus, comparing with \cite[Theorem 1.4]{H1993}, we see once again that it is advantageous to define
${\bf D}_{q,\omega}$ as in Definition \ref{def-Fqwu-Foup}--(iii).
\end{remark}

\begin{remark}
It may seems somehow intricate the way how formulas (\ref{EqDistC3-Dqw1}) and (\ref{EqDistC3-Dqw2})
appear on the course of the proof of Theorem \ref{Dqw-main-Thm}. In fact, they were given
in the proof of the sufficiency of the condition,
hence without assuming {\it a priori} the regularity of ${\bf u}$
(as a matter of fact, they were used to prove the regularity of ${\bf u}$).
Assuming the regularity of ${\bf u}$, there is a more transparent way to obtain those formulas.
Indeed, going back to the end of the proof of the necessity of the condition on Theorem \ref{Dqw-main-Thm},
we may deduce (\ref{EqDistC3-Dqw1}) and (\ref{EqDistC3-Dqw2}) as follows.
First, from (\ref{NuPaa}), we may write
$$
\gamma_j^{[n]}=\frac{\big\langle{\bf u}^{[n]},\big(P_{j}^{[n]}\big)^2\big\rangle}{\big\langle{\bf u}^{[n]},\big(P_{j-1}^{[n]}\big)^2\big\rangle}
=\frac{q^{-n}[j]_qd_{j+2n-2}}{[j+n]_qd_{j+n-2}}
\frac{\langle{\bf u},P_{j+n}^2\rangle}{\langle{\bf u},P_{j+n-1}^2\rangle}
=\frac{q^{-n}[j]_qd_{j+2n-2}}{[j+n]_qd_{j+n-2}}\gamma_{j+n}
$$
for every $j=1,2,\ldots$ and $n=0,1,\ldots$.
Taking $j=1$ and using (\ref{g1n}), we obtain
$$
\gamma_{n+1}=\frac{q^{n}[n+1]_qd_{n-1}}{d_{2n-1}}\gamma_1^{[n]}
=-\frac{q^{n}[n+1]_qd_{n-1}}{d_{2n-1}d_{2n+1}}\phi\left(-\frac{e_n}{d_{2n}}\right)\;.
$$
This proves (\ref{EqDistC3-Dqw2}).
To prove (\ref{EqDistC3-Dqw1}), set
$$
P_n^{[k]}(x)=x^n+t_n^{[k]}x^{n-1}+(\mbox{\rm lower degree terms})\;,
$$
for each $k=0,1,\ldots$.
It is well known (see e.g. \cite[Theorem 4.2-(d)]{C1978}) that
$$
t_n^{[k]}=-\sum_{j=0}^{n-1}\beta_j^{[k]}\quad(k=0,1,\ldots;\;n=1,2,\ldots)\;.
$$
Using (\ref{Dqwxn}), and recalling that $P_n^{[0]}=P_n$, we deduce
\begin{align*}
D_{q,\omega}P_{n+1}(x) & = D_{q,\omega}(x^{n+1})+t_{n+1}^{[0]}D_{q,\omega}(x^n)+(\mbox{\rm lower degree terms})\\
& =[n+1]_qx^n+\left\{\big((n+1)[n]_q-n[n+1]_q\big)\omega_0+t_{n+1}^{[0]}[n]_q\right\}x^{n-1} \\
& \quad +(\mbox{\rm lower degree terms})\;,
\end{align*}
hence, since $P_n^{[1]}(x):=D_{q,\omega}P_{n+1}(x)/[n+1]_q$, we obtain
$$
t_n^{[1]} = \Big(\frac{(n+1)[n]_q}{[n+1]_q}-n\Big)\omega_0+t_{n+1}^{[0]}\frac{[n]_q}{[n+1]_q}\quad(n=1,2,\ldots)\;.
$$
Rewrite this equality as
$$
\frac{t_{n+1}^{[0]}+(n+1)\omega_0}{[n+1]_q}=\frac{t_n^{[1]}+n\omega_0}{[n]_q}\quad(n=1,2,\ldots)\;.
$$
Applying successively this relation, yields
$$
\frac{t_{n+1}^{[0]}+(n+1)\omega_0}{[n+1]_q}=\frac{t_1^{[n]}+1\cdot\omega_0}{[1]_q}
=-\beta_0^{[n]}+\omega_0\quad(n=1,2,\ldots)\;,
$$
hence
$$
t_{n+1}^{[0]}=\big([n+1]_q-(n+1)\big)\omega_0-[n+1]_q\beta_0^{[n]}\quad(n=0,1,\ldots)\;.
$$
(Note that this equality is trivial if $n=0$.)
Therefore,
$$
\beta_n=\beta_n^{[0]}=t_{n}^{[0]}-t_{n+1}^{[0]}
=\big([n]_q-[n+1]_q+1\big)\omega_0+[n]_q\beta_0^{[n-1]}-[n+1]_q\beta_0^{[n]}\;.
$$
This proves (\ref{EqDistC3-Dqw1}), since $\beta_0=u_1/u_0=-e/d$, hence $\beta_0^{[n]}=-e_n/d_{2n}$,
and taking into account that $([n]_q-[n+1]_q+1)\omega_0=[n]_q\omega$.
\end{remark}

\begin{remark}\label{Rmk7}
Suppose that ${\bf u}\in\mathcal{P}^*$ is regular and satisfies the functional equation (\ref{EqDistC1-Dqw}).
Then the Rodrigues-type formula (\ref{EqDistRod-Dqw}) is a simple consequence of the relation between the dual basis
$({\bf a}_n)_{n\geq0}$ and $\big({\bf a}_n^{[k]}\big)_{n\geq0}$
associated to the monic OPS $(P_n)_{n\geq0}$ and $(P_n^{[k]})_{n\geq0}$ ($k=0,1,\ldots$), respectively.
To see why this holds we first observe that until now we only have made use of the space $\mathcal{P}^*$,
the algebraic dual of $\mathcal{P}$.
Consider now $\mathcal{P}$ endowed with the strict inductive limit topology induced by the
spaces $\mathcal{P}_n$ ($n=0,1,\ldots$), each $\mathcal{P}_n$ being regarded as a finite dimensional normed space.
Then, denoting by $\mathcal{P}'$ the topological dual of $\mathcal{P}$,
the equality $\mathcal{P}^*=\mathcal{P}'$ holds (see e.g. \cite{M1991,P2018}).
As a consequence, we may write (in the sense of the weak dual topology in $\mathcal{P}'$):
$$
{\bf D}_{1/q,-\omega/q}^k\big({\bf a}_n^{[k]}\big)
=\sum_{j=0}^\infty\langle{\bf D}_{1/q,-\omega/q}^k\big({\bf a}_n^{[k]}\big),P_j\rangle{\bf a}_j
\quad(n=0,1,\ldots)\;.
$$
Since $\langle{\bf D}_{1/q,-\omega/q}^k\big({\bf a}_n^{[k]}\big),P_j\rangle=0$ if $j<k$ and,
if $j\geq k$,
$$
\langle{\bf D}_{1/q,-\omega/q}^k\big({\bf a}_n^{[k]}\big),P_j\rangle
=(-q)^k\langle{\bf a}_n^{[k]},D_{q,\omega}^kP_j\rangle
=(-q)^k\frac{[j]_q!}{[j-k]_q!}\langle{\bf a}_n^{[k]},P_{j-k}^{[k]}\rangle\;,
$$
we deduce
$$
{\bf D}_{1/q,-\omega/q}^k\big({\bf a}_n^{[k]}\big)
=(-q)^k\,\frac{[n+k]_q!}{[n]_q!}\,{\bf a}_{n+k}\quad(n,k=0,1,\ldots)\;.
$$
Taking $n=0$ and then replacing $k$ by $n$, we obtain
$$
{\bf D}_{1/q,-\omega/q}^n\big({\bf a}_0^{[n]}\big)
=(-q)^n[n]_q!{\bf a}_{n}\quad(n=0,1,\ldots)\;.
$$
Therefore, since
${\bf a}_0^{[n]}=\langle{\bf u}^{[n]},1\rangle^{-1}{\bf u}^{[n]}$ and
${\bf a}_{n}=\langle{\bf u},P_n^2\rangle^{-1}P_n{\bf u}$ (see \cite{M1991,M1994}), we deduce
$$
{\bf D}_{1/q,-\omega/q}^n\big({\bf u}^{[n]}\big)
=(-q)^n[n]_q!\frac{\langle{\bf u}^{[n]},1\rangle}{\langle{\bf u},P_n^2\rangle} P_{n}{\bf u}\quad(n=0,1,\ldots)\;.
$$
Finally, taking into account (\ref{repres-unk}) and (\ref{NuPaa}), (\ref{EqDistRod-Dqw}) follows.
\end{remark}

\section*{Acknowledgements}

The authors thank the Stuttgart University Library for kindly sending them a
hard copy of H\"acker's Ph.D. Thesis \cite{H1993} together with the preprint \cite{H1993a}.

\section*{Disclosure statement}

No potential conflict of interest was reported by the authors.

\section*{Funding}

RAN is supported by Ministerio de Econom\'{\i}a y Competitividad of Spain
through MTM2015-65888-C4-1-P, Junta de Andaluc\'ia through FQM-262, and Feder Funds (European Union).
KC, DM, and JP are partially supported by the Centre for Mathematics of the University of Coimbra -- UID/MAT/00324/2019, funded by the Portuguese Government through FCT/MEC and co-funded by the European Regional Development Fund through the Partnership Agreement PT2020. DM is also supported by the FCT grant PD/BD/135295/2017.

\end{document}